\documentclass[11pt,a4paper]{article}
\usepackage[utf8]{inputenc}
\usepackage[english]{babel}
\usepackage{amssymb, amsmath, enumerate}
\usepackage{graphicx}
\usepackage{epstopdf}
\usepackage{subcaption}
\usepackage{authblk}
\usepackage{textcomp}

\newtheorem{theorem}{Theorem}[section]

\newtheorem{proposition}[theorem]{Proposition}
\newtheorem{corollary}[theorem]{Corollary}

\newenvironment{proof}[1][]{\begin{trivlist}
		\item[\hskip \labelsep {\bfseries #1}]}{\end{trivlist}}

\newcommand{\qed}{\nobreak \ifvmode \relax \else
	\ifdim\lastskip<1.5em \hskip-\lastskip
	\hskip1.5em plus0em minus0.5em \fi \nobreak
	\vrule height0.75em width0.5em depth0.25em\fi}
\begin{document}
	\title{Game-theoretical model of cooperation between producers in a production process: 3-agent interaction case}
	\author[1]{O.~A. Malafeyev\thanks{malafeyevoa@mail.ru}}
	\author[1]{A.~P. Parfenov\thanks{parf@bk.ru}}
	\affil[1]{Saint-Petersburg State University,  Russia}
	\date{}
	\maketitle
	\begin{abstract}
		A network model of manufacturing system is considered. This is a network formation game where players are participants of a production process and their actions are their's requests for interaction. Production networks are formed as a result of an interaction. Players' payoff functions are defined on the set of all possible networks. In this paper the special case of network formation games is considered. Payoff functions are supposed to be additive and depend on subsets of arcs. Two cases are considered. First, subsets of arcs are supposed to be not intersected. The necessary and sufficient conditions for equilibrium are given for this case. The second case is the one where subsets of arcs are determined by 3-agent coalitions. An illustrative example is given where equilibria and a compromise solution are found.
	\end{abstract}
	
	

	\section{Introduction}
	
In production processes a problem of partner selection often arises. The problems of this kind are studied in the mathematical theory of network games. Network games are the games where each strategy profile forms a network (i.e. graph with weighted nodes and arcs). Network games of various types are considered in \cite{Petros1, GrigorievaPostman, GrigorievaPostmanRus, NetGamesJackson, NetGames, MalafeevCorruptionNet, Petros2}. They can be applied to supply chain modelling \cite{SupplyChainNetGame, SupplyChainHierGame}.

Network formation games \cite{NetGamesGubko, NetGamesJackson, NetGames} are games where strategy of a player is a set of offers to other players for interaction. In such game each strategy profile forms a network where a set of vertices is a set of players, and arcs are pairs of players interacting with each other.

Similar models have been intensively studied in recent years \cite{NetGames, Petros1}. A model with asymmetric offers was also considered \cite{NetGamesGubko}. In such game strategy profile forms a directed graph. As a rule, a payoff function on a set of graphs is an arbitrary real-valued function. A set of Nash equilibrium points for a network formation game is always non-empty. However the problem of finding equilibria may be hard enough because one needs to check a large number of strategies. This number grows faster than exponent of a number of players \cite{_NetGamesFormation}.

For some classes of network games, rather simple necessary and sufficient conditions for the strategy profiles to be equilibria one can find in \cite{NetGames}. Thus it is a good idea to study particular cases of payoff functions. For example, in \cite{NetGamesSosnina3} additive payoff functions defined by 3-agent coalitions are considered.  

In this paper particular cases of network formation games are studied. Results of this study can be used to model manufacturing systems. A manufacturing system with finite number of agents is considered. This system can produce, transport or store some products. A tuple of agent's product flows is called stable if each agent can`t increase his profit by changing its own flows. So stable flows are Nash equilibrium points in the network game.

Network games with additive value functions for the case of coalitions with limited number of players are considered as well. Necessary and sufficient conditions for strategy profile to be Nash equilibria are found when coalitions do not intersect. Network games with 3-agent coalitions are analyzed afterwards. An example is given. For this example Nash equilibria and a compromise solution are calculated.

Papers related to the theme of this article are \cite{b1}--\cite{52}. 
	
\section{Additive payoff functions determined by network subgraphs}

Each agent in a production system gets profits and pays costs from production chains that he uses. Net profit of an agent is a sum of  his profits and losses. 

Let
$$
S' \subseteq 2^M
$$
be a set of possible network subgraphs. Each subgraph is determined by a set of arcs of network. Let's denote by $S'(g)$ a set of subgraphs from $S'$ that are subgraphs of $g$: 
$$
S'(g) = S' \cap 2^{M(g)} = \{M' \in S' \mid M' \subseteq M(g)\}.
$$
sufficient conditions for a network to be stable are satisfied both for the case of pairwise intersecting and for the case of disjoint subgraphs

Let $D(S)$ be an income from agents' interaction determined by subgraph $S \in S'(g)$. Each agent $x$ gets a part $\alpha_{x}(S) \geq 0$ from this income. Thus
$$
\sum_{x \in S} \alpha_x(S)=1.
$$

Let $H_{x}(\varphi)$ be a total profit of agent $x$ in action profile $\varphi$. It is a sum of profits of agent $x$ in all subgraphs in action profile $\varphi$. Let
$$
S_{\varphi}(x)= \{S \in S'(g(\varphi)) \mid x \in S\}
$$
be a set of all subgraphs in action profile $\varphi$ that contain the agent $x$. Then the total profit of the agent $x$ is as follows
$$
H_x(\phi) = \sum_{S \in S_{\varphi}(x)}\alpha_{x}(S)D(S).
$$

Thus we have described the game
$$
G_{S'} = \{ N,\Phi_G, \{D(S)\}_{S \in S'}, \{\alpha_x(S)\}_{S \in S', \ x \in S}) \}.
$$

\begin{proposition}
	\label{UtvStability}
	Suppose any 2 subgraphs from $S'$ do not intersect. Suppose the network $g$ does not have loops. Then the network $g$ is stable iff for all its subgraphs $S \in S'(g) \ D(S) \geq 0$.
\end{proposition}
\begin{proof}
	Necessity. Let's prove the proposition by contradiction. Suppose that there exists a subgraph $S$ such that $D(S) < 0$. For players in $S$  we have $\sum_{x \in S} \alpha_x(S)=1$. Thus there exists a player $x \in S$ such that $\alpha_x(S) > 0$. If $x$ breaks link with another player $y$ in $S$ then $S$ disappears from $S'(g)$. Other subgraphs still exist in $g$, because there is no pairwise intersections of subgraphs. Thus payoff of the player $x$ increases by $-\alpha_x(S)D(S)$. But the network is stable. This is a contradiction.
	
	Now let's prove the sufficiency. Let $x$ be a player. If $x$ breaks any links in action profile $\varphi$, a set of subgraphs in $S_{\varphi}(x)$ decreases. Thus the player $x$ does not increase its payoff (because for any subgraph $S \ D(S) \geq 0$).
\end{proof}
\begin{corollary}
	Suppose a payoff function of each player $x$ is additive and is determined by 2-player coalitions (i.e. only by links of $x$):
	$$
	H_x(g) = \sum_{(x, y) \in M(g)}\alpha_{(x, y)}(x)D(x, y) + \sum_{(y, x) \in M(g)}\alpha_{(y, x)}(x)D(y, x).
	$$
	
	Then the network $g$ is stable iff for all $x, y \ D(x, y) \geq 0$.
\end{corollary}
\begin{proof}
	The corollary is true because subgraps of 2-agent coalitions do not pairwise intersect.
\end{proof}

Sufficient conditions for a network to be stable are satisfied both for the case of pairwise intersecting and for the case of disjoint subgraphs.

\section{Additional payoff function determined by 2-agent and 3-agent subgraphs}

Profit of agent $j$ depends on its contracts with each agent $i \in N$. In addition, profit may depend on other factors --- for example, on contracts with both agent $i$ and $k$ simultaneously.

Examples:
\begin{enumerate}
	\item agent $i$ and agent $k$ supply two kinds of components; agent $j$ makes final production using these components;
	
	\item agent $j$ moves products from agent $i$ to agent $k$;
	
	\item agent $j$ is a mediator: agent $j$ buys products from agent $i$ and sells it to agent $k$.
\end{enumerate}

Thus coalitions with size greater than 2 must be considered. Agents' profit may depend on coalitions of such kind.

Earlier games with payoffs depending on 3-player coalitions were considered in \cite{NetGamesSosnina3}. However often incomes and costs depend both on 2-player and 3-player coalitions.

Coalition $S=(i_1, i_2, i_3) \subset N$ arises if each agent $i \in S$ wishes to cooperate with all other agents $j \in S \setminus \{i\}$ and agrees with offers from all other agents $j \in S \setminus \{i\}$. The player $i$'s payoff depends on two kinds of subgraphs of the network: subgraphs determined by 2-player coalitions
$$
S_2 = \{(i, j), (j, i)\}_{i,j \in N}
$$
and by 3-player coalitions
$$
S_3 = \{(i, j), (j, i), (i, k), (k, i), (j, k), (k, j)\}_{i,j,k \in N}.
$$

Coalitions with 2 and 3 players in action profile $\varphi$ are enumerated as follows:
$$
S_2(g(\varphi)) \cup S_3(g(\varphi)) = \{ S_{\varphi}^{1}, S_{\varphi}^{2}, \ldots, S_{\varphi}^{n_\phi}\}.
$$

Proposition \ref{UtvStability} on network stability condition may not hold true because 3-player coalitions may be pairwise intersected.

{\bf Example}. The case of 5 agents that can form any links. Their weights are as follows $\alpha_i(S) = 1/3, \ i = 1, \dots, 5$. Let coalitions' payoffs be:
$$
\begin{array}{l}
D(1, 2, 3) = 2\\
D(1, 3, 4) = -1\\
D(1, 4, 5) = 2\\
D(3, 4, 5) = 2
\end{array}
$$
and $D(S) = 0$ for other coalitions.

The network $g$ with edges
$$
\begin{array}{l}
M(g) = \{(1, 2), (2, 1), (1, 3), (3, 1), (1, 4), (4, 1), (1, 5), (5, 1), \\
(2, 3), (3, 2), (3, 4), (4, 3), (3, 5), (5, 3), (4, 5), (5, 4)\}
\end{array}
$$
is stable. However this network has coalition $S = (1, 3, 4)$ for which $D(S) = -1$. If  any link in this coalition is broken, then the coalition $S$ disappears from the network but one of the coalitions $(1, 2, 3), (1, 4, 5), (3, 4, 5)$ disappears as well. These coalitions produce payoff $D(S) = 2$ greater than $-D(1, 3, 4)$ and thus the network is stable.

\section{Equilibrium and compromise solution computing}

Let's take Nash equilibrium and compromise solution \cite{Malafeev} as optimality concepts  in the game $G$.

To find equilibria and compromise solutions, it is convenient to construct payoff matrix in the game $G$. Rows of this matrix are action profiles in $G$, columns are the players.

To find equilibria, consider each network $g$ and check whether it is stable. The checking procedure for agreement game is as follows: for each player $i$ deviation from action profile $\varphi$ that breaks some links, check whether the payoff $H_i$ in this action profile is greater than $H_i(g)$.

Compromise solution can be found by simple algorithm as follows:
\begin{enumerate}
	\item Calculate the ideal vector $M=(M_1, \dots, M_n)$, where $M_i = \max_{\varphi}(H_i(\varphi))$.
	\item For each action profile $\varphi \in \Phi$ for each player $i \in N$ calculate differences $M_i - H_i(\varphi), \ i \in N$.
	\item For each action profile $\varphi \in \Phi$ calculate maximal difference $\max_i(M_i - H_i(\varphi))$.
	\item Find minimum of maximal differences over the set $\Phi$ of all action profiles
	$$
	\min_{\varphi \in \Phi}\max_i(M_i - H_i(\varphi)).
	$$
\end{enumerate}

Action profile where minimum is achieved is a compromise solution of game $G$.

\section{The example}

A model of 2-player and 3-player interaction with 5 players in a network is considered. Lets denote them $1, 2, 3, 4, 5$. Assume possible action profiles are $\{\varphi_s\}_{s=1}^{10}$. Each action profile is determined by 2 matrices $\Gamma^+$ and $\Gamma^-$.

For the action profile $\varphi_1$ matrices $\Gamma^+, \Gamma^-$ and the matrix of the network $g = \min(\Gamma^+, {\Gamma^-}^T)$ are as follows
$$
\Gamma^+_1 = \left(\begin{array}{ccccc}
0& 1& 1& 1& 1\\
0& 0& 0& 0& 0\\
0& 1& 0& 1& 0\\
0& 1& 0& 0& 0\\
0& 1& 0& 1& 0
\end{array}\right), \ \Gamma^-_1 = \left(\begin{array}{ccccc}
0& 0& 0& 0& 0\\
0& 0& 0& 0& 1\\
1& 0& 0& 0& 0\\
1& 0& 1& 0& 1\\
1& 0& 0& 0& 0
\end{array}\right), \ g_1 = \left(\begin{array}{ccccc}
0& 0& 1& 1& 1\\
0& 0& 0& 0& 0\\
0& 0& 0& 1& 0\\
0& 0& 0& 0& 0\\
0& 1& 0& 1& 0
\end{array}\right).
$$

Similarly for the action profiles $\varphi_2, \dots, \varphi_{10}$ matrices $\Gamma^+, \Gamma^-$ and matrices of the networks $g = \min(\Gamma^+, {\Gamma^-}^T)$ are as follows
$$
\Gamma^+_2 = \left(\begin{array}{ccccc}
0& 0& 1& 1& 1\\
1& 0& 0& 0& 0\\
0& 1& 0& 1& 1\\
0& 1& 0& 0& 1\\
0& 1& 0& 0& 0
\end{array}\right), \ \Gamma^-_2 = \left(\begin{array}{ccccc}0& 0& 0& 0& 0\\
0& 0& 1& 1& 1\\
1& 0& 0& 0& 0\\
1& 0& 1& 0& 0\\
1& 0& 1& 0& 0
\end{array}\right), \ g_2 = \left(\begin{array}{ccccc}
0& 0& 1& 1& 1\\
0& 0& 0& 0& 0\\
0& 1& 0& 1& 1\\
0& 1& 0& 0& 0\\
0& 1& 0& 0& 0
\end{array}\right).
$$
$$
\Gamma^+_3 = \left(\begin{array}{ccccc}
0& 0& 1& 0& 0\\
1& 0& 0& 0& 0\\
0& 1& 0& 0& 0\\
1& 1& 0& 0& 0\\
0& 1& 0& 1& 0
\end{array}\right), \ \Gamma^-_3 = \left(\begin{array}{ccccc}
0& 1& 0& 0& 0\\
0& 0& 1& 1& 1\\
1& 0& 0& 0& 0\\
0& 0& 1& 0& 1\\
1& 0& 1& 0& 0
\end{array}\right), \ g_3 = \left(\begin{array}{ccccc}
0& 0& 1& 0& 0\\
1& 0& 0& 0& 0\\
0& 1& 0& 0& 0\\
0& 1& 0& 0& 0\\
0& 1& 0& 1& 0
\end{array}\right).
$$
$$
\Gamma^+_4 = \left(\begin{array}{ccccc}
0& 0& 0& 1& 1\\
1& 0& 0& 1& 0\\
1& 1& 0& 1& 1\\
0& 0& 0& 0& 1\\
0& 1& 0& 0& 0
\end{array}\right), \ \Gamma^-_4 = \left(\begin{array}{ccccc}
0& 0& 0& 0& 0\\
0& 0& 1& 0& 1\\
0& 0& 0& 0& 0\\
1& 1& 1& 0& 0\\
1& 0& 0& 1& 0
\end{array}\right), \ g_4 = \left(\begin{array}{ccccc}
0& 0& 0& 1& 1\\
0& 0& 0& 1& 0\\
0& 1& 0& 1& 0\\
0& 0& 0& 0& 1\\
0& 1& 0& 0& 0
\end{array}\right).
$$
$$
\Gamma^+_5 = \left(\begin{array}{ccccc}
0& 1& 1& 1& 1\\
0& 0& 0& 0& 0\\
0& 1& 0& 1& 1\\
0& 1& 0& 0& 1\\
0& 1& 0& 0& 0
\end{array}\right), \ \Gamma^-_5 = \left(\begin{array}{ccccc}
0& 0& 0& 0& 0\\
1& 0& 1& 1& 1\\
1& 0& 0& 0& 0\\
1& 0& 1& 0& 0\\
1& 0& 1& 1& 0
\end{array}\right), \ g_5 = \left(\begin{array}{ccccc}
0& 1& 1& 1& 1\\
0& 0& 0& 0& 0\\
0& 1& 0& 1& 1\\
0& 1& 0& 0& 1\\
0& 1& 0& 0& 0
\end{array}\right).
$$
$$
\Gamma^+_6 = \left(\begin{array}{ccccc}
0& 1& 0& 1& 0\\
0& 0& 0& 0& 0\\
0& 1& 0& 0& 0\\
0& 1& 1& 0& 0\\
0& 1& 0& 1& 0
\end{array}\right), \ \Gamma^-_6 = \left(\begin{array}{ccccc}
0& 0& 0& 0& 0\\
1& 0& 0& 1& 1\\
0& 0& 0& 0& 0\\
1& 0& 0& 0& 0\\
1& 0& 0& 0& 0
\end{array}\right), \ g_6 = \left(\begin{array}{ccccc}
0& 1& 0& 1& 0\\
0& 0& 0& 0& 0\\
0& 0& 0& 0& 0\\
0& 1& 0& 0& 0\\
0& 1& 0& 0& 0
\end{array}\right).
$$
$$
\Gamma^+_7 = \left(\begin{array}{ccccc}
0& 0& 1& 1& 0\\
0& 0& 0& 1& 1\\
0& 1& 0& 0& 1\\
0& 0& 0& 0& 1\\
0& 0& 0& 0& 0
\end{array}\right), \ \Gamma^-_7 = \left(\begin{array}{ccccc}
0& 0& 0& 0& 1\\
0& 0& 1& 0& 0\\
1& 0& 0& 0& 0\\
0& 1& 0& 0& 0\\
0& 1& 0& 1& 0
\end{array}\right), \  g_7 = \left(\begin{array}{ccccc}
0& 0& 0& 0& 0\\
0& 0& 0& 1& 1\\
0& 1& 0& 0& 0\\
0& 0& 0& 0& 1\\
0& 0& 0& 0& 0
\end{array}\right).
$$
$$
\Gamma^+_8 = \left(\begin{array}{ccccc}
0& 1& 1& 1& 1\\
0& 0& 1& 0& 1\\
0& 0& 0& 0& 0\\
0& 0& 0& 0& 0\\
0& 0& 1& 1& 0
\end{array}\right), \ \Gamma^-_8 = \left(\begin{array}{ccccc}
0& 0& 0& 0& 0\\
0& 0& 0& 0& 0\\
0& 1& 0& 0& 1\\
1& 0& 0& 0& 1\\
1& 1& 0& 0& 0
\end{array}\right), \ g_8 = \left(\begin{array}{ccccc}
0& 0& 0& 1& 1\\
0& 0& 1& 0& 1\\
0& 0& 0& 0& 0\\
0& 0& 0& 0& 0\\
0& 0& 1& 1& 0
\end{array}\right).
$$
$$
\Gamma^+_9 = \left(\begin{array}{ccccc}
0& 0& 0& 0& 0\\
0& 0& 0& 0& 1\\
1& 1& 0& 0& 0\\
1& 1& 1& 0& 1\\
1& 0& 0& 0& 0
\end{array}\right), \Gamma^-_9 = \left(\begin{array}{ccccc}
0& 0& 0& 0& 0\\
0& 0& 1& 1& 0\\
0& 0& 0& 0& 0\\
0& 0& 0& 0& 0\\
0& 1& 0& 1& 0
\end{array}\right), \ g_9 = \left(\begin{array}{ccccc}
0& 0& 0& 0& 0\\
0& 0& 0& 0& 1\\
0& 1& 0& 0& 0\\
0& 1& 0& 0& 1\\
0& 0& 0& 0& 0
\end{array}\right).
$$
$$
\Gamma^+_{10} = \left(\begin{array}{ccccc}
0& 0& 0& 0& 0\\
0& 0& 0& 1& 1\\
0& 1& 1& 1& 0\\
1& 0& 0& 0& 1\\
0& 1& 0& 0& 0
\end{array}\right), \ \Gamma^-_{10} = \left(\begin{array}{ccccc}
0& 0& 0& 0& 0\\
0& 0& 1& 0& 1\\
0& 0& 0& 0& 0\\
0& 1& 1& 0& 0\\
0& 0& 1& 1& 0
\end{array}\right), \ g_{10} = \left(\begin{array}{ccccc}
0& 0& 0& 0& 0\\
0& 0& 0& 1& 0\\
0& 1& 0& 1& 0\\
0& 0& 0& 0& 1\\
0& 1& 0& 0& 0
\end{array}\right).
$$

Income functions and weights of the players ${1, 2, 3, 4, 5}$ for each 3-players coalition are as follows
$$
\begin{array}{l}
S^1 = (1, 3, 4), \ D(S^1) = 4, \ \alpha_{1}={\frac 12}, \ \alpha_{3}={\frac 14}, \ \alpha_{4}={\frac 14}\\
S^2 = (2, 5, 4), \ D(S^2) = 3, \ \alpha_{2}={\frac 13}, \ \alpha_{5}={\frac 13}, \ \alpha_{4}={\frac 13}\\
S^3 = (1, 4, 5), \ D(S^3) = 6, \ \alpha_{1}={\frac 13}, \ \alpha_{4}={\frac 13}, \ \alpha_{5}={\frac 13}\\
S^4 = (1, 3, 5), \ D(S^4) = 8, \ \alpha_{1}={\frac 12}, \ \alpha_{3}={\frac 14}, \ \alpha_{5}={\frac 14}\\
S^5 = (3, 4, 2), \ D(S^5) = 4, \ \alpha_{3}={\frac 14}, \ \alpha_{4}={\frac 14}, \ \alpha_{2}={\frac 12}\\
S^6 = (1, 2, 5), \ D(S^6) = 12, \ \alpha_{1}={\frac 24}, \ \alpha_{2}={\frac 14}, \ \alpha_{5}={\frac 14}\\
S^7 = (1, 4, 2), \ D(S^7) = 8, \ \alpha_{1}={\frac 14}, \ \alpha_{4}={\frac 12}, \ \alpha_{2}={\frac  12}\\
S^8 = (3, 4, 5), \ D(S^8) = 18, \ \alpha_{3}={\frac 13}, \ \alpha_{4}={\frac 13}, \ \alpha_{5}={\frac 13}\\
S^9 = (3, 5, 2), \ D(S^9) = 16, \ \alpha_{3}={\frac 12}, \ \alpha_{5}={\frac 14}, \ \alpha_{2}={\frac  14}\\
S^{10} = (3, 2, 1), \ D(S^{10}) = 21, \ \alpha_{3}={\frac 13}, \ \alpha_{2}={\frac 13}, \ \alpha_{1}={\frac  13}
\end{array}
$$
where $S^i$ is a coalition numbered by index $i$ and $D(S^i)$ is its profit.

For 2-players coalitions we get
$$
\begin{array}{l}
S^{11} = (1, 3), \ D(S^{11}) = -2, \ \alpha_{1} = \alpha_{3} = {\frac 12}\\
S^{12} = (1, 4), \ D(S^{12}) = -6, \ \alpha_{1} = \alpha_{4} = {\frac 12}
\end{array}
$$
and $D = 0$ for other 2-players coalitions.

Payoffs for action profiles ${\{\varphi_s\}_{s=1}^{10}}$ are as follows
$$
\begin{array}{l}
{H_1(\varphi_{1})= 0},  {H_2(\varphi_{1})=1}, {H_3(\varphi_{1})=0},
{H_4(\varphi_{1})=1}, {H_5(\varphi_{1})= 3};\\

{H_1(\varphi_{2})=2}, {H_2(\varphi_{2})=2},
{H_3(\varphi_{2})= 3}, {H_4(\varphi_{2})= -1}, {H_5(\varphi_{2})= 2};\\

{H_1(\varphi_{3})= 6}, {H_2(\varphi_{3})= 6}, {H_3(\varphi_{3})= 2}, {H_4(\varphi_{3})= 2}, {H_5(\varphi_{3})=6};\\

{H_1(\varphi_{4})=-1}, {H_2(\varphi_{4})= 3}, {H_3(\varphi_{4})= 1}, {H_4(\varphi_{4})= 1}, {H_5(\varphi_{4})= 3};\\

{H_1(\varphi_{5})= 19}, {H_2(\varphi_{5})!}, {H_3(\varphi_{5})= 25}, {H_4(\varphi_{5})= 12}, {H_5(\varphi_{5})= 22};\\

{H_1(\varphi_{6})= 5}, {H_2(\varphi_{6})= 9}, {H_3(\varphi_{6})= 1},
{H_4(\varphi_{6})= 2 }, {H_5(\varphi_{6})= 3};\\

{H_1(\varphi_{7})= 0}, {H_2(\varphi_{7})= 1},
{H_3(\varphi_{7})= 0}, {H_4(\varphi_{7})= 1 }, {H_5(\varphi_{7})= 1};\\

{H_1(\varphi_{8})= -1}, {H_2(\varphi_{8})= 4}, {H_3(\varphi_{8})= 8},
{H_4(\varphi_{8})= -1 }, {H_5(\varphi_{8})= 6};\\

{H_1(\varphi_{9})= 2}, {H_2(\varphi_{9})= 1}, {H_3(\varphi_{9})= 1},
{H_4(\varphi_{9})= 2}, {H_5(\varphi_{9})= 1};\\

{H_1(\varphi_{10})= 0}, {H_2(\varphi_{10})= 3}, {H_3(\varphi_{10})= 7},
{H_4(\varphi_{10})= 8}, {H_5(\varphi_{10})= 7}.
\end{array}
$$

Thus payoff matrix for players $1, 2, 3, 4, 5$ for action profiles $s = 1, \ldots, 10$ is
$$
\left(\begin{array}{ccccc}
4& 1& 1& 4& 3\\
6& 2& 4& 2& 2\\
7& 6& 3& 2& 6\\
2& 3& 1& 4& 3\\
23& 21& 26& 15& 22\\
8& 9& 1& 5& 3\\
0& 1& 0& 1& 1\\
2& 4& 8& 2& 6\\
2& 1& 1& 2& 1\\
0& 3& 7& 8& 7
\end{array}\right)
$$


There is only one action profile that can be changed to another action profile by 1-player action change. It is the profile $\varphi_4$ that can be changed to $\varphi_{10}$ by player 1 breaking his links with players 4 and 5. Payoff of the player 1 for the profile $\varphi_4$ is -1 and for the profile $\varphi_{10}$ is 0. Thus the profile $\varphi_4$ is not a Nash equilibrium. Other profiles are Nash equilibria. 

The ideal vector is as follows $ M = ( 23, 21, 26, 15,22)$.

For each action profile $\varphi_s \in \Phi_G$ for each player $i \in N$ let's calculate differences $M_i - H_i(\varphi_s), \ i \in N$. These differences (sorted in ascending order) for each action profile $\varphi_s$ are:
$$
\begin{array}{l}
M_i - H_i(\varphi_{1})= \triangle_{1}=(11, 19, 19, 20, 25)\\
M_i - H_i(\varphi_{2})= \triangle_{2}=(13, 17,19, 20, 22)\\
M_i - H_i(\varphi_{3})= \triangle_{3}=(13, 15,16, 16, 23)\\
M_i - H_i(\varphi_{4})= \triangle_{4}=(11, 18, 19, 21, 25)\\
M_i - H_i(\varphi_{5})= \triangle_{5}=(0, 0, 0, 0, 0)\\
M_i - H_i(\varphi_{6})= \triangle_{6}=(10, 12, 15, 19, 25)\\
M_i - H_i(\varphi_{7})= \triangle_{7}=(14, 20, 21, 23, 26)\\
M_i - H_i(\varphi_{8})= \triangle_{8}=(13, 16, 18, 21, 27)\\
M_i - H_i(\varphi_{9})= \triangle_{9}=(13, 20, 21, 21, 25)\\
M_i - H_i(\varphi_{10})= \triangle_{10}=(7, 15, 18, 19, 23).
\end{array}
$$

Minimum of maximal differences over action profiles
$$
\min_{ \varphi_G \in \Phi_G } \max_{i \in N}(M_i - H_i(\varphi_s))
$$
is a value of a compromise solution. For this example the compromise solution is an action profile $\varphi_{5}$.

\section{Conclusion}

Network formation games with additive payoff functions determined by subgraphs are considered. For the case where subgraphs do not pairwise intersect, conditions for strategy profile to be equilibrium are formulated. They are also conditions for network to be stable. A network is stable iff for any coalition with negative payoff the network does not consist some coalition links in this network. 

In the case where some subgraphs intersect, this condition is only sufficient. In this case some stable networks one can find, but not all stable networks can be found in this case. 

In the case where subgraphs are determined by 3-player coalitions the algorithm finding equilibria and stable networks is proposed. The example with 5 players is considered. For this example equilibria and compromise solution are found.

For further research the case of additive payoffs determined by 3-player coalitions is interesting. It is the unsolved problem to find necessary and sufficient conditions for network to be stable in this case.
\newline
	\section{Acknowledgements}
	The work is partly supported by work RFBR No. 18-01-00796.

\end{document}